\documentclass[11pt]{article}
\usepackage[top=2.5cm,bottom=2.5cm,left=2.5cm,right=2.5cm]{geometry}
\usepackage{amssymb}
\usepackage{amsmath,amsthm}
\usepackage[latin1]{inputenc}
\usepackage[dvips]{graphicx}
\usepackage{hyperref}
\usepackage{color}
\usepackage{mathrsfs}
\usepackage{enumerate}
\usepackage{tikz}
\usepackage{xifthen}
\usepackage{verbatim}
\usepackage{soul}

\hypersetup{colorlinks=true}

\hypersetup{colorlinks=true, linkcolor=blue, citecolor=blue,urlcolor=blue}

\newtheorem{remark}{Remark}[section]

\newtheorem{lemma}[remark]{Lemma}
\newtheorem{theorem}[remark]{Theorem}
\newtheorem{proposition}[remark]{Proposition}
\newtheorem{observation}[remark]{Observation}

\title{Relating the outer-independent total Roman domination number with some classical parameters of graphs}

\author{Abel Cabrera Mart\'inez$^{(1)}$, Dorota Kuziak $^{(2)}$ and Ismael G. Yero$^{(3)}$\\
\\
$^{(1)}$ {\small Universitat Rovira i Virgili,}
{\small Departament d'Enginyeria Inform\`atica i Matem\`atiques } \\  {\small Av. Pa\"{\i}sos
Catalans 26, 43007 Tarragona, Spain.} \\{\small
  abel.cabrera\@@urv.cat}\\
$^{(2)}${\small Departamento de Estad\'istica e Investigaci\'on Operativa, Escuela T\'ecnica Superior de Ingenier\'ia de Algeciras}\\
{\small Universidad de C\'adiz,} {\small
Av. Ram\'on Puyol s/n, 11202 Algeciras, Spain.} \\ {\small
 dorota.kuziak\@@uca.es}\\
$^{(3)}${\small Departamento de Matem\'aticas, Escuela T\'ecnica Superior de Ingenier\'ia de Algeciras}\\
{\small Universidad de C\'adiz,} {\small
Av. Ram\'on Puyol s/n, 11202 Algeciras, Spain.} \\ {\small
ismael.gonzalez\@@uca.es}
}

\date{ }
\begin{document}
\maketitle

\begin{abstract}
For a given graph $G$ without isolated vertex we consider a function $f: V(G) \rightarrow \{0,1,2\}$. For every $i\in \{0,1,2\}$, let $V_i=\{v\in V(G):\; f(v)=i\}$. The function $f$ is known to be an outer-independent total Roman dominating function for the graph $G$ if it is satisfied that; (i) every vertex \textcolor[rgb]{0.00,0.00,0.00}{in $V_0$} is adjacent to at least one vertex in $V_2$; (ii) $V_0$ is an independent set; and (iii) the subgraph induced by $V_1\cup V_2$ has no isolated vertex. The minimum possible weight $\omega(f)=\sum_{v\in V(G)}f(v)$ among all outer-independent total Roman dominating functions for $G$ is called the outer-independent total Roman domination number of $G$. In this article we obtain new tight bounds for this parameter that improve some well-known results. Such bounds can also be seen as relationships between this parameter and several other classical parameters in graph theory like the domination, total domination, Roman domination, independence, and vertex cover numbers. In addition, we compute the outer-independent total Roman domination number of Sierpi\'nski graphs, circulant graphs, and the Cartesian and direct products of complete graphs.
\end{abstract}

{\it Keywords}: outer-independent total Roman domination; vertex cover; domination; (total) Roman domination; independence.

{\it AMS Subject Classification Numbers:} 05C69.

\section{Introduction and preliminaries} \label{Intro}

This work mainly deals with showing some existent interconnections between several classical graph theory topics like domination, independence, covers, and Roman domination. These topics have attracted the attention of several researches in the last few decades, and a high number of significant contributions are nowadays well known. Each new topic or parameter that is described naturally gives more insight into the classical ones, and it is usually welcome by the research community. A quick search in databases like MathSciNet or similar ones will show a large number of works, both theoretical and applied, and a wide number of researchers that have dealt with them, and are indeed still making so. \textcolor[rgb]{0.00,0.00,0.00}{Our goal is to making} some new remarkable contributions to these topics through mixing some of them, or relating or bounding some ones with the others.

We begin by stating the main basic terminology which shall be used in the whole exposition. We first consider a non directed graph $G$ without loops or multiple edges. For a vertex $x\in V(G)$, by $N_G(x)$ we mean the \emph{open neighborhood} of $x$, \emph{i.e.}, the set of vertices of $G$ adjacent to $x$. The \emph{closed neighborhood} of $x$ is then $N_G[x]=N_G(x)\cup \{x\}$. \textcolor[rgb]{0.00,0.00,0.00}{If the graph $G$ is clear from the context, then we remove the subindexes in the notations above.} The \emph{minimum} and \emph{maximum degrees} of $G$ are $\delta(G)$ and $\Delta(G)$, respectively (or $\delta$ and $\Delta$ if $G$ is clear from the context). A vertex of degree one in $G$ is a \emph{leaf}, and a vertex adjacent to a leaf is a \emph{support vertex}. The set of leaves and support vertices of $G$ are denoted by $L(G)$ and $S(G)$, respectively.

A set of vertices is an \emph{independent set} of $G$, if it induces a subgraph without edges. The maximum possible cardinality of an independent set of $G$ is the \emph{independence number} of $G$, and is denoted by $\beta(G)$. In some kind of ``oposed'' side of an independent set, we find a \emph{vertex cover set}, which is a set of vertices of $G$ such that every edge of $G$ is incident to at least one vertex of such set. The minimum possible cardinality of a vertex cover set is the \emph{vertex cover number} of $G$, denoted by $\alpha(G)$. It is well-known that for any graph $G$ of order $n$ it follows that $\alpha(G)+\beta(G)=n$ (Gallai's theorem).

A set of vertices $D\subset V(G)$ is a \emph{dominating set} of $G$, if  every vertex $v\in V(G)\setminus D$ satisfies that $N(v)\cap D\ne\emptyset$. The \textcolor[rgb]{0.00,0.00,0.00}{minimum possible} cardinality among all dominating sets of $G$ is the \emph{domination number} of $G$, which is denoted by $\gamma(G)$. Similarly, the set $D$ is a \emph{total dominating set} of $G$, if every vertex $v\in V(G)$ satisfies that $N(v)\cap D\ne\emptyset$. The \emph{total domination number}, $\gamma_t(G)$, of $G$ is analogously defined. Studies on domination and total domination in graphs \textcolor[rgb]{0.00,0.00,0.00}{are two of the most} commonly found in the literature. The books \cite{book-dom-1,book-dom-2,book-total-dom} represent fairly complete compendiums of results in these topics (although the first two of them are not much updated by now). A variant of domination (and indeed of total domination also) is that one mixing total domination properties with vertex independence. That is, a total dominating set $D$ is called an \emph{outer-independent total dominating} set of $G$, if $V(G)\setminus D$ is an independent set. The \emph{outer-independent total domination number} of $G$ is also analogously to the domination number defined, and is denoted by $\gamma_{t,oi}(G)$. This parameter was introduced and barely studied in \cite{Soner2012}, and further well studied in \cite{paper-TOID-products,Cabrera2017a}.

An important variant of domination in \textcolor[rgb]{0.00,0.00,0.00}{graphs} is that of Roman domination, which saw its formal birthday in \cite{Cockayne2004}, due to some kind of historical reasons arising from the ancient Roman Empire, that were described in the works \cite{Revelle2000,Stewart1999}. We consider a function $f:V(G)\rightarrow \{0,1,2\}$. For a set \textcolor[rgb]{0.00,0.00,0.00}{$S\subseteq V(G)$}, the \emph{weight} of $S$ under $f$ is $f(S)=\sum_{v\in S}f(v)$. If $S=V(G)$, then the weight of $S$ is indeed called the weight of $f$, and denoted by $\omega(f)$, \emph{i.e.}, $\omega(f)=f(V(G))=\sum_{v\in V(G)}f(v)$. Clearly, any such function determines three sets of vertices that we denote by $V_0,V_1,V_2$. That is, $V_i=\{v\in V(G)\,:\,f(v)=i\}$. Since there is a one to one relation between $f$ and the sets it determines, from now on we shall write $f(V_0,V_1,V_2)$ to better refer to our function $f$. With these concepts in mind, a function $f(V_0,V_1,V_2)$ is called a \emph{Roman dominating function} on $G$, if every $v\in V_0$ has a neighbor $u\in V_2$. The \emph{Roman domination number} of $G$ is then the minimum possible weight among all Roman dominating functions on $G$, and is denoted by $\gamma_R(G)$. To see the relationship between this concept and the historical situation of the Roman Empire, we suggest the seminal article \cite{Cockayne2004}, although there are nowadays several interesting references on this topic that also explain and improve this relationship by making some specifications.

One of the attempts on improving the ideas of Roman domination in graphs was first presented in \cite{LiCh13} through some more general settings, and after formally, specifically and better studied in \cite{total-Roman-1,paper-TRDF-2019}. The main idea of such improvement comes with the addition of a total domination property. That is, a Roman dominating function $f(V_0,V_1,V_2)$ is called a \emph{total Roman dominating function} on $G$, if the subgraph induced by $V_1\cup V_2$ has no isolated vertices, \emph{i.e.}, $V_1\cup V_2$ is a total dominating set of $G$. The \emph{total Roman domination number} of $G$ is analogously defined, and denoted by $\gamma_{tR}(G)$.

Another improvements of the Roman domination concept are that ones connecting them with independent sets, and thereby, with vertex cover sets. A (total) Roman dominating function $f(V_0,V_1,V_2)$ is called an \emph{outer-independent $($total$)$ Roman dominating function} (OIRDF and OITRDF for short) if $V_0$ is an independent set of $G$. Notice that, this is equivalent to say that $V_1\cup V_2$ is a vertex cover set of $G$. In connection with this last fact, it would be even more natural to call such function a covering (total) Roman dominating function instead of outer-independent (total) Roman dominating function. However, in order to keep the already stated terminology, we prefer to use that on OIRD and OITRD functions. The \emph{outer-independent $($total$)$ Roman domination number} of $G$ is the minimum possible weight among all outer-independent (total) Roman dominating functions on $G$, and is denoted by ($\gamma_{oitR}(G)$) $\gamma_{oiR}(G)$. The parameter $\gamma_{oiR}(G)$ was introduced in \cite{Abdollahzadeh2017a}, while $\gamma_{oitR}(G)$ was first presented in \cite{Cabrera2019}.

Once defined all the concepts above, we are prepared to begin with our exposition. Our main objective is then to present several relationships between all the parameters mentioned above, by making emphasis on $\gamma_{oitR}(G)$, which is the center of our investigation. We must remark that this parameter has also been recently studied in \cite{Li2018,Mojdeh2019}. For instance, in \cite{Li2018}, some computational and approximation results on this parameter were presented, and in \cite{Mojdeh2019}, some Nordhauss-Gaddum results for it were proved. In our work, we bound $\gamma_{oitR}(G)$ in terms of the other above mentioned parameters, give several chains of inequalities involving many of these parameters, and finally, we present exact values of it for a number of remarkable families of graphs $G$ which have been recently attracting the attention of several researchers. From now on, for a parameter $p(G)$ of a graph $G$, by a $p(G)$-\emph{set}, or a $p(G)$-\emph{function}, we mean a set of cardinality $p(G)$ or a function of weight $p(G)$, respectively.

\subsection{Some primary and basic results}

In order to be used as examples in several places, for showing the tightness or not of several bounds and relationships, we give some preliminary results in this subsection concerning a few basic families of graphs. Some of them are classical ones in studies of domination in graphs, although in such cases, the computations are straightforward to see, and thus left to the reader.

\begin{remark}
\label{rem:bipartite}
For any complete bipartite graph $K_{r,s}$, with $3\le r\le s$, the following observations hold.
\begin{itemize}
  \item $\gamma_{R}(K_{r,s})=\gamma_{tR}(K_{r,s})=4$.
  \item $\gamma_{oiR}(K_{r,s})=\gamma_{t,oi}(K_{r,s})=r+1$.
   \item $\gamma_{oitR}(K_{r,s})=r+2$.
\end{itemize}
\end{remark}

The \emph{wheel graph} $W_n$ is a graph of order $n$ formed by connecting a single universal vertex to all vertices of a cycle of order $n-1$.

\begin{remark}
\label{rem:wheel}
For any wheel graph $W_n$, with $n\ge 4$, the following observations hold.
\begin{itemize}
  \item $\alpha(W_n)=\left\lceil\frac{n-1}{2}\right\rceil+1$.
  \item $\gamma_{oiR}(W_n)=\left\lceil\frac{n-1}{2}\right\rceil+2$.
  \item $\gamma_{t,oi}(W_n)=\left\lceil\frac{n-1}{2}\right\rceil+1$.
  \item $\gamma_{oitR}(W_n)=\left\lceil\frac{n-1}{2}\right\rceil+2$.
\end{itemize}
\end{remark}

The family $\mathcal{F}_{p,q}$ of graphs, which we next construct, shall also be useful for our purposes. We begin with $p$ star graphs $S_{1,t_1}, \ldots , S_{1,t_p}$, with centers $c_1,\dots,c_p$, respectively, such that $t_1,\dots,t_p\ge 3$; and $q$ complete bipartite graphs $K_{r_1,r'_1}$, $\dots$, $K_{r_q,r'_q}$ with $4\le r_i\le r'_i/2$ for every $i\in\{1,\dots,q\}$. Next, for every $i\in\{1,\dots,q\}$, we add $r_i$ pendant vertices to exactly two vertices, say $x_i,y_i$, of each complete bipartite graph $K_{r_i,r'_i}$ belonging to the bipartition set of cardinality $r_i$. Finally, to obtain a graph $G\in \mathcal{F}_{p,q}$, we add an extra vertex $w$, and join $w$ with an edge to exactly one leaf, say $z_i$, of each star $S_{1,t_i}$, and to exactly one vertex, say $w_i$, of the bipartition set of cardinality $r'_i$ of each complete bipartite graph $K_{r_i,r'_i}$. Figure \ref{F-2-2} shows a fairly representative example of a graph in $\mathcal{F}_{2,2}$.

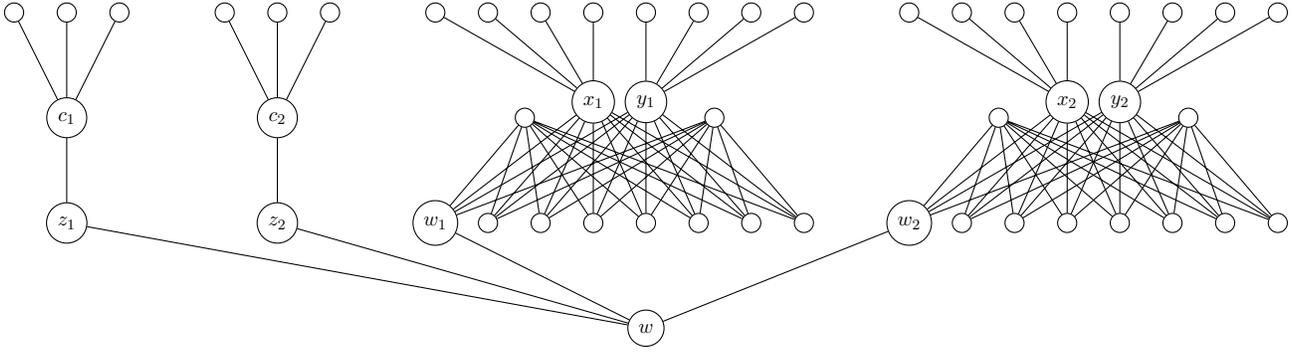
\begin{figure}[h]
\centering
\begin{tikzpicture}[scale=.7, transform shape]
\node [draw, shape=circle] (a1) at  (1,0) {$z_1$};
\node [draw, shape=circle] (a2) at  (1,4) {};
\node [draw, shape=circle] (a3) at  (2,4) {};
\node [draw, shape=circle] (a4) at  (0,4) {};

\node [draw, shape=circle] (b1) at  (5,0) {$z_2$};
\node [draw, shape=circle] (b2) at  (5,4) {};
\node [draw, shape=circle] (b3) at  (6,4) {};
\node [draw, shape=circle] (b4) at  (4,4) {};

\node [draw, shape=circle] (c1) at  (8,0) {$w_1$};
\node [draw, shape=circle] (c2) at  (9,0) {};
\node [draw, shape=circle] (c3) at  (10,0) {};
\node [draw, shape=circle] (c4) at  (11,0) {};
\node [draw, shape=circle] (c5) at  (12,0) {};
\node [draw, shape=circle] (c6) at  (13,0) {};
\node [draw, shape=circle] (c7) at  (14,0) {};
\node [draw, shape=circle] (c8) at  (15,0) {};

\node [draw, shape=circle] (d1) at  (17,0) {$w_2$};
\node [draw, shape=circle] (d2) at  (18,0) {};
\node [draw, shape=circle] (d3) at  (19,0) {};
\node [draw, shape=circle] (d4) at  (20,0) {};
\node [draw, shape=circle] (d5) at  (21,0) {};
\node [draw, shape=circle] (d6) at  (22,0) {};
\node [draw, shape=circle] (d7) at  (23,0) {};
\node [draw, shape=circle] (d8) at  (24,0) {};

\node [draw, shape=circle] (a22) at  (1,2) {$c_1$};
\node [draw, shape=circle] (b22) at  (5,2) {$c_2$};

\node [draw, shape=circle] (c33) at  (9.7,2) {};
\node [draw, shape=circle] (c44) at  (11,2.3) {$x_1$};
\node [draw, shape=circle] (c55) at  (12,2.3) {$y_1$};
\node [draw, shape=circle] (c66) at  (13.3,2) {};
\node [draw, shape=circle] (d33) at  (18.7,2) {};
\node [draw, shape=circle] (d44) at  (20,2.3) {$x_2$};
\node [draw, shape=circle] (d55) at  (21,2.3) {$y_2$};
\node [draw, shape=circle] (d66) at  (22.3,2) {};

\node [draw, shape=circle] (c111) at  (8,4) {};
\node [draw, shape=circle] (c222) at  (9,4) {};
\node [draw, shape=circle] (c333) at  (10,4) {};
\node [draw, shape=circle] (c444) at  (11,4) {};
\node [draw, shape=circle] (c555) at  (12,4) {};
\node [draw, shape=circle] (c666) at  (13,4) {};
\node [draw, shape=circle] (c777) at  (14,4) {};
\node [draw, shape=circle] (c888) at  (15,4) {};
\node [draw, shape=circle] (d111) at  (17,4) {};
\node [draw, shape=circle] (d222) at  (18,4) {};
\node [draw, shape=circle] (d333) at  (19,4) {};
\node [draw, shape=circle] (d444) at  (20,4) {};
\node [draw, shape=circle] (d555) at  (21,4) {};
\node [draw, shape=circle] (d666) at  (22,4) {};
\node [draw, shape=circle] (d777) at  (23,4) {};
\node [draw, shape=circle] (d888) at  (24,4) {};

\node [draw, shape=circle] (w) at  (12,-2) {$w$};

\draw(a1)--(a22)--(a2);
\draw(a22)--(a3);
\draw(a22)--(a4);

\draw(b1)--(b22)--(b2);
\draw(b22)--(b3);
\draw(b22)--(b4);

\draw(c1)--(c33)--(c2)--(c44)--(c3)--(c55)--(c4)--(c66)--(c5)--(c55)--(c6)--(c66)--(c7)--(c55)--(c8)--(c66)--(c1)--(c44)--(c4)--(c33);
\draw(c33)--(c3)--(c66)--(c2)--(c55)--(c1);
\draw(c5)--(c33)--(c6)--(c44)--(c7)--(c33)--(c8)--(c44)--(c5);

\draw(c111)--(c44)--(c222);
\draw(c333)--(c44)--(c444);
\draw(c666)--(c55)--(c777);
\draw(c555)--(c55)--(c888);

\draw(d1)--(d33)--(d2)--(d44)--(d3)--(d55)--(d4)--(d66)--(d5)--(d55)--(d6)--(d66)--(d7)--(d55)--(d8)--(d66)--(d1)--(d44)--(d4)--(d33);
\draw(d33)--(d3)--(d66)--(d2)--(d55)--(d1);
\draw(d5)--(d33)--(d6)--(d44)--(d7)--(d33)--(d8)--(d44)--(d5);

\draw(d111)--(d44)--(d222);
\draw(d333)--(d44)--(d444);
\draw(d666)--(d55)--(d777);
\draw(d555)--(d55)--(d888);

\draw(a1)--(w)--(b1);
\draw(c1)--(w)--(d1);

\end{tikzpicture}
\caption{A graph of the family $\mathcal{F}_{2,2}$.}\label{F-2-2}
\end{figure}

The following remark gives the values of some domination parameters of a graph in $\mathcal{F}_{p,q}$. Some of these values can be straightforwardly computed, and thus left to the reader's discretion.

\begin{remark}
\label{rem:F-p-q}
For any graph $G\in \mathcal{F}_{p,q}$ the following claims hold.
\begin{itemize}
  \item[{\rm (a)}] $\gamma(G)=p+3q$.
  \item[{\rm (b)}] $\gamma_t(G)=2p+3q$.
  \item[{\rm (c)}] $\gamma_R(G)=2p+6q$.
  \item[{\rm (d)}] $\gamma_{tR}(G)=3p+6q$.
  \item[{\rm (e)}] $\gamma_{t,oi}(G)=2p+q+\sum_{i=1}^{q}r_i$.
  \item[{\rm (f)}] $\gamma_{oiR}(G)=2p+2q+1+\sum_{i=1}^{q}r_i$.
  \item[{\rm (g)}] $\gamma_{\textcolor[rgb]{0.00,0.00,0.00}{oitR}}(G)=3p+3q+\sum_{i=1}^{q}r_i$.
\end{itemize}
\end{remark}

\begin{proof}
(a) and (b) can be easily observed. For (c), we note that the function $f_3$ defined as follows,
$$f_3(u)=\left\{\begin{array}{ll}
                  2, & \mbox{if $u\in \{c_1,\dots,c_p\}\cup \{w_1,\dots,w_q\}\cup \{x_i,y_i\,:\,1\le i\le q\}$}, \\
                  0, & \mbox{otherwise},
                \end{array}
\right.$$
is a $\gamma_R(G)$-function with weight $2p+6q$. To observe (d), we use the function $f_4$ defined by
$$f_4(u)=\left\{\begin{array}{ll}
                  2, & \mbox{if $u\in \{c_1,\dots,c_p\}\cup \{w_1,\dots,w_q\}\cup \{x_i,y_i\,:\,1\le i\le q\}$}, \\
                  1, & \mbox{if $u\in \{z_1,\dots,z_p\}$},\\
                  0, & \mbox{otherwise},
                \end{array}
\right.$$
to get a $\gamma_{tR}(G)$-function with weight $3p+6q$. For item (e), we first note that the set $S=\{c_1,\dots,c_p\}\cup \{z_1,\dots,z_p\}\cup \{w_1,\dots,w_q\}$, together with the bipartition set of cardinality $r_i$ of each complete bipartite graph $K_{r_i,r'_i}$, forms a $\gamma_{t,oi}(G)$-set of cardinality $2p+q+\sum_{i=1}^{q}r_i$. Now, for (f), we observe that the function $f_6$ given as,
$$f_6(u)=\left\{\begin{array}{ll}
                  2, & \mbox{if $u\in \{c_1,\dots,c_p\}\cup \{x_i,y_i\,:\,1\le i\le q\}$}, \\
                  1, & \mbox{if $u=w$ or if $u\in W$},\\
                  0, & \mbox{otherwise},
                \end{array}
\right.$$
where $W$ is the union of all bipartition sets of cardinality $r_i$ of each complete bipartite graph $K_{r_i,r'_i}$ minus the vertices $x_i,y_i$,
is a $\gamma_{oiR}(G)$-function of weight $2p+2q+1+\sum_{i=1}^{q}r_i$. Finally, for item (g), we consider the function $f_7$ defined as
$$f_7(u)=\left\{\begin{array}{ll}
                  2, & \mbox{if $u\in \{c_1,\dots,c_p\}\cup \{x_i,y_i\,:\,1\le i\le q\}$}, \\
                  1, & \mbox{if $u\in \{z_1,\dots,z_p\}\cup \{w_1,\dots,w_q\}\cup W$},\\
                  0, & \mbox{otherwise},
                \end{array}
\right.$$
with $W$ as defined above. With some not so hard arguments, we note that such function $f_7$ is a $\gamma_{oitR}(G)$-function of weight $3p+3q+\sum_{i=1}^{q}r_i$.
\end{proof}

\section{Bounds and relationships with other parameters}

Cabrera Mart\'inez, Kuziak and Yero \cite{Cabrera2019} in 2019, established the following result for any connected nontrivial graph, although it also holds for any graph with no isolated vertex.

\begin{theorem}\label{teo-oitR-VC-Cabrera2019}{\rm \cite{Cabrera2019}}
For any graph $G$ with no isolated vertex,
$$\alpha(G)+1\leq \gamma_{oitR}(G)\leq 3\alpha(G).$$
\end{theorem}

In addition, they characterized \textcolor[rgb]{0.00,0.00,0.00}{the families of connected graphs} $G$  that satisfy the equalities $\gamma_{oitR}(G)=3\alpha(G)$ and $\gamma_{oitR}(G)=3\alpha(G)-1$.

\vspace{.2cm}

In order to improve these bounds above, we need to introduce the next results. Also, we recall that a graph is \emph{claw-free} if it does not contain $K_{1,3}$ as an induced subgraph.

\begin{observation}\label{obs-teo-principal}
For any graph $G$ with no isolated vertex, order $n$ and maximum degree $\Delta\geq 2$,
\begin{enumerate}
\item[{\rm (i)}]$1\leq \left\lceil\frac{n-\alpha(G)}{\Delta-1}\right\rceil$.
\item[{\rm (ii)}] $\gamma_t(G)\leq \gamma_R(G)\leq 2\gamma(G)\leq 2\alpha(G)$ {\rm(from \cite{HRSW}, \cite{Cockayne2004} and \cite{book-dom-1})}.
\end{enumerate}
\end{observation}

\begin{lemma}\label{lem:cover-total-dom-claw}
Let $G$ be a claw-free graph of minimum degree $\delta\geq 3$. If $S$ is a vertex cover of $G$, then $S$ is also a total dominating set of $G$.
\end{lemma}

\begin{proof}
Let $S$ be a vertex cover of $G$. Hence, $S$ is also a dominating set. If the subgraph induced by $S$ has an  isolated vertex $v$, then since $G$ has minimum degree three, the vertex $v$ has at least three neighbors not in $S$. Since $V(G)\setminus S$ is an independent set, we have that $v$ together with these three neighbors induce a $K_{1,3}$, which is not possible. Therefore, the subgraph induced by $S$ has no isolated vertices or equivalently, $S$ is a total dominating set of $G$.
\end{proof}

With the results above in mind, we state the following theorem, which improves the bounds given in Theorem \ref{teo-oitR-VC-Cabrera2019}.

\begin{theorem}\label{teo-principal}
For any graph $G$ with no isolated vertex, order $n$ and maximum degree $\Delta\geq 2$,
$$\alpha(G)+\max\left\{|S(G)|, \left\lceil\frac{n-\alpha(G)}{\Delta-1}\right\rceil\right\}\leq \gamma_{oitR}(G)\leq \alpha(G)+\gamma_t(G).$$
Moreover, for any claw-free graph $G$ of minimum degree $\delta\geq 3$,
$$\gamma_{oitR}(G)\leq \alpha(G)+\gamma(G).$$
\end{theorem}

\begin{proof}
We first proceed to prove the first upper bound. Let $D$ be a $\gamma_t(G)$-set and $S$ an $\alpha(G)$-set. Let $f(V_0,V_1,V_2)$ be a function defined by $V_0=V(G)\setminus (D\cup S)$, $V_1=(D\cup S)\setminus (D\cap S)$ and $V_2= D\cap S$. We claim that $f$ is an OITRDF on $G$.

It is straightforward that $V_0=V(G)\setminus (D\cup S)$ is an independent set and $V_1\cup V_2=D\cup S$ is a total dominating set. We only need to prove that every vertex in $V_0$ has a neighbor in $V_2$. Let $x\in V_0=V(G)\setminus (D\cup S)$. Since $S$ is a vertex cover and $D$ is a total dominating set, we deduce that $N(x)\subseteq S$ and $N(x)\cap D\neq \emptyset$, respectively. Hence $N(x)\cap D\cap S\neq \emptyset$, or equivalently, $N(x)\cap V_2\neq  \emptyset$. Thus, $f$ is an OITRDF on $G$, as desired, and so, $\gamma_{oitR}(G)\leq \omega(f)\leq |(D\cup S)\setminus (D\cap S)|+2|D\cap S|=\alpha(G)+\gamma_t(G)$, as desired.

Next, we proceed to prove the lower bound. Let $f(V_0,V_1,V_2)$ be a $\gamma_{oitR}(G)$-function. Notice first that $(V_1\setminus L(G))\cup V_2$ is a vertex cover. Moreover, $S(G)\subseteq V_1\cup V_2$ and $| V_1\cap S(G)|\leq |V_1\cap L(G)|$.  Therefore,
\begin{eqnarray*}
\gamma_{oitR}(G) &= &|V_1|+2|V_2|\\
  &=&|V_1\setminus L(G)|+|V_1\cap L(G)|+2|V_2|\\
  &=&(|V_1\setminus L(G)|+|V_2|)+(|V_1\cap L(G)|+|V_2|)\\
  &\geq&(|V_1\setminus L(G)|+|V_2|)+(|V_1\cap S(G)|+|V_2|)\\
  &\geq &\alpha(G)+|S(G)|.
\end{eqnarray*}

Now, we notice that every vertex in $V_2$ has at most $\Delta-1$ neighbors in $V_0$ as $V_1\cup V_2$ is a total dominating set. Hence, $|V_0|\leq (\Delta-1)|V_2|$. Taking into account the inequality above, and the fact that \textcolor[rgb]{0.00,0.00,0.00}{$n-\alpha(G)\geq |V(G)\setminus (V_1\cup V_2)|=|V_0|$} as $V_1\cup V_2$ is a vertex cover, we have
\begin{eqnarray*}
(\Delta-1)\gamma_{oitR}(G) &= &(\Delta-1)(|V_1|+2|V_2|)\\
  &=&(\Delta-1)(|V_1|+|V_2|)+(\Delta-1)|V_2|\\
  &\geq & (\Delta-1)(n-|V_0|)+|V_0|\\
  &\geq & \textcolor[rgb]{0.00,0.00,0.00}{n(\Delta-1)} -(\Delta-2)|V_0|\\
  &\geq & \textcolor[rgb]{0.00,0.00,0.00}{n(\Delta-1)} -(\Delta-2)(n-\alpha(G))\\
  &= & (\Delta-1)\alpha(G)+ (n-\alpha(G)).
\end{eqnarray*}

This implies that $\gamma_{oitR}(G)\geq \alpha(G)+\left\lceil\frac{n-\alpha(G)}{\Delta-1}\right\rceil$, which completes the first part of the proof.

We now consider $G$ is claw-free. The next proof follows along the lines of \textcolor[rgb]{0.00,0.00,0.00}{the first part of} this proof. We assume $D$ is a $\gamma(G)$-set and $S$ is an $\alpha(G)$-set. We claim that the function $f(V_0,V_1,V_2)$ (as above) is an OITRDF on $G$. Recall that $V_0=V(G)\setminus (D\cup S)$, $V_1=(D\cup S)\setminus (D\cap S)$ and $V_2= D\cap S$.

It is clear that $V_0$ is an independent set as $V_0\subseteq V(G)\setminus S$. Since $G$ is claw-free graph, by Lemma \ref{lem:cover-total-dom-claw}, we have that $S$ is also a total dominating set. Hence, $V_1\cup V_2=D\cup S$ is a total dominating set as well. We only need to prove that every vertex in $V_0$ has a neighbor in $V_2$. Let $x\in V_0=V(G)\setminus (D\cup S)$. Since $S$ is a vertex cover and $D$ is a dominating set, we deduce that $N(x)\subseteq S$ and $N(x)\cap D\neq \emptyset$, respectively. Hence $N(x)\cap D\cap S\neq \emptyset$, \emph{i.e.}, $N(x)\cap V_2\neq \emptyset$. Therefore, $f$ is an OITRDF on $G$, as desired. Thus, $\gamma_{oitR}(G)\leq \omega(f)\leq |(D\cup S)\setminus (D\cap S)|+2|D\cap S|=\alpha(G)+\gamma(G)$, which completes the proof.
\end{proof}

The bounds above are tight. For instance, for the graph $G$ shown in Figure \ref{fig-1} we have that $\gamma_{oitR}(G)=7=\alpha(G)+\gamma_t(G)$. Also, the complete graph $K_n$ satisfies that $\gamma_{oitR}(K_n)=n=\alpha(K_n)+\left\lceil\frac{n-\alpha(K_n)}{\Delta-1}\right\rceil$. Furthermore, in \cite{Cabrera2019}, the authors showed that the corona graph $G\cong G_1\odot N_r$ satisfies that $\gamma_{oitR}(G)=2|S(G)|=\alpha(G)+|S(G)|$. In addition, for the case of wheel graphs $W_n$, since they have no leaves, they clearly have no support vertices, and so $|S(W_n)|=0$. For such graphs, the lower bound above is tight when $n\ge 7$, since $\left\lceil\frac{n-\alpha(\textcolor[rgb]{0.00,0.00,0.00}{W_n})}{\Delta(W_n)-1}\right\rceil=1$ and, by Remark \ref{rem:wheel}, $\gamma_{oitR}(W_n)=\left\lceil\frac{n-1}{2}\right\rceil+2=\alpha(W_n)+1$. Other graphs that show the tightness of the upper bound of Theorem \ref{teo-principal} are the complete bipartite graphs $K_{r,s}$, which can be seen by using Remark \ref{rem:bipartite}. For the tightness of the bound concerning claw-free graphs, we consider for instance the complete graph $K_n$ ($n\geq 4$), which is claw-free, and satisfies that $\gamma_{oitR}(K_n)=n=\alpha(K_n)+\gamma(K_n)$.

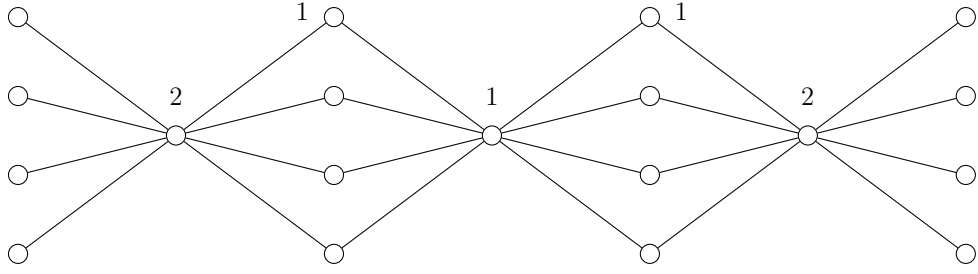
\begin{figure}[ht]
\centering
\begin{tikzpicture}[scale=.7, transform shape]
\node [draw, shape=circle] (a1) at  (0,0) {};
\node [draw, shape=circle] (a2) at  (0,1.5) {};
\node [draw, shape=circle] (a3) at  (0,3) {};
\node [draw, shape=circle] (a4) at  (0,4.5) {};

\node [draw, shape=circle] (b1) at  (6,0) {};
\node [draw, shape=circle] (b2) at  (6,1.5) {};
\node [draw, shape=circle] (b3) at  (6,3) {};
\node [draw, shape=circle] (b4) at  (6,4.5) {};

\node [draw, shape=circle] (c1) at  (12,0) {};
\node [draw, shape=circle] (c2) at  (12,1.5) {};
\node [draw, shape=circle] (c3) at  (12,3) {};
\node [draw, shape=circle] (c4) at  (12,4.5) {};

\node [draw, shape=circle] (d1) at  (18,0) {};
\node [draw, shape=circle] (d2) at  (18,1.5) {};
\node [draw, shape=circle] (d3) at  (18,3) {};
\node [draw, shape=circle] (d4) at  (18,4.5) {};

\node [draw, shape=circle] (ab) at  (3,2.25) {};
\node at (3,3) {\Large $2$};
\node [draw, shape=circle] (bc) at  (9,2.25) {};
\node at (9,3) {\Large $1$};
\node [draw, shape=circle] (cd) at  (15,2.25) {};
\node at (15,3) {\Large $2$};

\node at (5.4,4.6) {\Large $1$};
\node at (12.6,4.6) {\Large $1$};

\draw (a1)--(ab)--(b1)--(bc)--(c1)--(cd)--(d1);
\draw (a2)--(ab)--(b2)--(bc)--(c2)--(cd)--(d2);
\draw (a3)--(ab)--(b3)--(bc)--(c3)--(cd)--(d3);
\draw (a4)--(ab)--(b4)--(bc)--(c4)--(cd)--(d4);

\end{tikzpicture}
\caption{A graph $G$ with $\gamma_{oitR}(G)=7=\alpha(G)+\gamma_t(G)$ with the positive labels of \textcolor[rgb]{0.00,0.00,0.00}{a} $\gamma_{oitR}(G)$-function.}\label{fig-1}
\end{figure}

The next result, which also improves the upper bound given in Theorem \ref{teo-oitR-VC-Cabrera2019}, is an immediate consequence of Theorem \ref{teo-principal} and Observation \ref{obs-teo-principal} (ii).

\begin{theorem}\label{teo-cons-teo-principal}
For any graph $G$ with no isolated vertex,
$$\gamma_{oitR}(G)\leq \alpha(G)+\gamma_R(G)\leq \alpha(G)+2\gamma(G).$$
\end{theorem}

With respect to the equality in the bound $\gamma_{oitR}(G)\leq \alpha(G)+2\gamma(G)$ above, we can deduce the following connection. To this end, we need to say that a graph $G$ is called a \emph{Roman graph} if $\gamma_R(G)=2\gamma(G)$.

\begin{proposition}
If $G$ is a graph such that $\gamma_{oitR}(G)=\alpha(G)+2\gamma(G)$, then $G$ is a Roman graph.
\end{proposition}

\begin{proof}
From Theorem \ref{teo-cons-teo-principal}, we have that $\alpha(G)+2\gamma(G)=\gamma_{oitR}(G)\leq \alpha(G)+\gamma_R(G)\leq \alpha(G)+2\gamma(G)$. Thus, we must have equalities in the whole inequality chain. In particular, we conclude that $\gamma_R(G)=2\gamma(G)$, and so, $G$ is a Roman graph.
\end{proof}

Notice that the opposed to the result above is not necessarily true. For instance, any complete bipartite graph $K_{r,s}$, with $3\le r\le s$, is a Roman graph, but it does not satisfy the equality since, by Remark \ref{rem:bipartite}, $\alpha(K_{r,s})+2\gamma(K_{r,s})=r+4\ne r+2=\gamma_{oitR}(K_{r,s})$.

Concerning the outer-independent Roman domination number and the \textcolor[rgb]{0.00,0.00,0.00}{outer-independent total} domination number, which are closely related to $\gamma_{oitR}(G)$, in \cite{Cabrera2019}, the authors showed the following results.

\begin{theorem}\label{teo-oitR-oiR-toid-Cabrera2019}{\rm \cite{Cabrera2019}}
The following statements hold for any graph $G$ of order $n\geq 3$ with no isolated vertex.
\begin{enumerate}
\item[{\rm (i)}] $\gamma_{t,oi}(G)+1\leq \gamma_{oitR}(G)\leq 2\gamma_{t,oi}(G)$.
\item[{\rm (ii)}] If $f(V_0,V_1,V_2)$ is a $\gamma_{oiR}(G)$-function, then $\gamma_{oiR}(G)\leq \gamma_{oitR}(G)\leq \gamma_{oiR}(G)+|V_1|+|V_2|$.
\end{enumerate}
\end{theorem}

We now provide a result which improves the upper bounds given in Theorem \ref{teo-oitR-oiR-toid-Cabrera2019}.

\begin{theorem}\label{teo-oitR-oiR-toid}
For any graph $G$ with no isolated vertex,
$$\gamma_{oitR}(G)\leq \min\{\gamma_{t,oi}(G),\gamma_{oiR}(G)\}+\gamma(G).$$
\end{theorem}

\begin{proof}
First, we proceed to prove that $\gamma_{oitR}(G)\leq \gamma_{t,oi}(G)+\gamma(G)$. Let $D$ be a $\gamma_{t,oi}(G)$-set and $S$ a $\gamma(G)$-set. Let $f(V_0,V_1,V_2)$ be a function defined by $V_0=V(G)\setminus (D\cup S)$, $V_1=(D\cup S)\setminus (D\cap S)$ and $V_2=D\cap S$. We claim that $f$ is an OITRDF on $G$.
Notice that $V_0\subseteq V(G)\setminus D$ is an independent set and $V_1\cup V_2$ is a total dominating set as $D$ is a \textcolor[rgb]{0.00,0.00,0.00}{outer-independent total} dominating set of $G$. Now, we prove that every vertex in $V_0$ has a neighbor in $V_2$. Let $x\in V_0=V(G)\setminus (D\cup S)$. Since $D$ is also a vertex cover and $S$ is a dominating set, we deduce that $N(x)\subseteq D$ and $N(x)\cap S\neq \emptyset$, respectively. Thus, $N(x)\cap D\cap S\neq \emptyset$, i.e., $N(x)\cap V_2\neq  \emptyset$. Hence, $f$ is an OITRDF on $G$, as required. Thus, $\gamma_{oitR}(G)\leq \omega(f)\leq |(D\cup S)\setminus (D\cap S)|+2|D\cap S|=\gamma_{t,oi}(G)+\gamma(G)$, as desired.

Finally, we proceed to prove that $\gamma_{oitR}(G)\leq \gamma_{oiR}(G)+\gamma(G)$. In this case, let $g(W_0,W_1,W_2)$ be a $\gamma_{oiR}(G)$-function and $S$ a $\gamma(G)$-set. Now, we define a function $f(V_0,V_1,V_2)$ as follows.
\begin{enumerate}[{\rm (i)}]
\item $V_2=W_2$ and $W_1\subseteq V_1$.
\item If $x\in (W_1\cup W_2)\cap S$, then choose a vertex $y\in N(x)\cap W_0$ (\textcolor[rgb]{0.00,0.00,0.00}{if it exists}) and set $y\in V_1$.
\item If $x\in W_0\cap S$, then set $x\in V_1$.
\item For any  other vertex $x\in W_0$ not previously labelled, set $x\in V_0$.
\end{enumerate}
We claim that $f$ is an OITRDF on $G$.
Since $W_0$ is \textcolor[rgb]{0.00,0.00,0.00}{an} independent set and $V_0\subseteq W_0\subseteq N(W_2)$, it follows by (i) that $V_0$ is also an independent set and $V_0\subseteq N(V_2)$. Finally, by (ii), (iii), (iv) and the fact that $S$ is a dominating set, we deduce that $V_1\cup V_2$ is a total dominating set of $G$. Hence, $f$ is an OITRDF on $G$, as desired. Therefore, $\gamma_{oitR}(G)\leq \omega(f)=2|V_2|+|V_1|\leq 2|W_2|+|W_1|+|S|= \gamma_{oiR}(G)+\gamma(G)$, which completes the proof.
\end{proof}

The following result shows a class of graphs which satisfy the equality $\gamma_{oitR}(G)=\gamma_{oiR}(G)$.

\begin{theorem}
For any claw-free graph $G$ of minimum degree $\delta\geq 3$,
$$\gamma_{oitR}(G)=\gamma_{oiR}(G).$$
\end{theorem}

\begin{proof}
Let $f(V_0,V_1,V_2)$ be a $\gamma_{oiR}(G)$-function. Since $V_0$ is an independent set, we have that $V_1\cup V_2$ is a vertex cover of $G$. As every vertex cover in a claw-free graph of minimum degree $\delta\geq 3$ is also a total \textcolor[rgb]{0.00,0.00,0.00}{dominating set by Lemma \ref{lem:cover-total-dom-claw}}, we deduce that $f$ is also an OITRDF. Hence, $\gamma_{oitR}(G)\leq \omega(f)=\gamma_{oiR}(G)$. Theorem \ref{teo-oitR-oiR-toid-Cabrera2019} (ii) completes the proof.
\end{proof}

Moreover, by Theorems \ref{teo-oitR-oiR-toid-Cabrera2019} and \ref{teo-oitR-oiR-toid}, we deduce that the graphs $G$ with $\gamma_{oitR}(G)>\gamma_{oiR}(G)$ satisfy the following  inequality chain.
\begin{equation}\label{ineq-1}
\max\{\gamma_{t,oi}(G),\gamma_{oiR}(G)\}+1\leq \gamma_{oitR}(G)\leq \min\{\gamma_{t,oi}(G),\gamma_{oiR}(G)\}+\gamma(G).
\end{equation}

For instance, for \textcolor[rgb]{0.00,0.00,0.00}{the complete bipartite graph} $K_{1,n-1}$ we obtain equalities   through the previous inequality chain, i.e., $\gamma_{t,oi}(K_{1,n-1})+1=\gamma_{oiR}(K_{1,n-1})+1= \gamma_{oitR}(K_{1,n-1})=\gamma_{t,oi}(K_{1,n-1})+\gamma(K_{1,n-1})=\gamma_{oiR}(K_{1,n-1})+\gamma(K_{1,n-1})$. In contrast with the example above, if we consider $G\in \mathcal{F}_{p,q}$ (as defined in Section~\ref{Intro}), we note that there are graphs achieving a strict inequality in all the steps of the \textcolor[rgb]{0.00,0.00,0.00}{inequality} chain (\ref{ineq-1}). That is, for any graph $G\in \mathcal{F}_{p,q}$, (by using Remark \ref{rem:F-p-q}) it follows that
\begin{align*}
  \gamma_{t,oi}(G)+1=1+2p+q+\sum_{i=1}^{q}r_i & < \gamma_{oiR}(G)+1=1+2p+2q+\sum_{i=1}^{q}r_i\\
   & <\gamma_{oitR}(G)= 3p+3q+\sum_{i=1}^{q}r_i\\
   & <\gamma_{t,oi}(G)+\gamma(G)=2p+q+\sum_{i=1}^{q}r_i+p+3q=3p+4q+\sum_{i=1}^{q}r_i\\
   & < \gamma_{oiR}(G)+\gamma(G)=2p+2q+\sum_{i=1}^{q}r_i+p+3q=3p+5q+\sum_{i=1}^{q}r_i.
\end{align*}

\begin{theorem}\label{teo-paper-TOkID}{\rm\cite{teo-paper-TOkID}}
For any connected graph $G$ of  minimum degree $\delta$, $$\gamma_{t,oi}(G)\leq 2\alpha(G)-\delta+1.$$
\end{theorem}

Next, we provide a new upper bound for the outer-independent total Roman domination number, which is an immediate consequence of Theorems \ref{teo-oitR-oiR-toid} and \ref{teo-paper-TOkID}. Notice that this result improves Theorem \ref{teo-cons-teo-principal} for the graphs $G$ that satisfy the inequality $\alpha(G)\leq \gamma(G)+\delta-1$.

\begin{theorem}
For any connected graph $G$ of minimum degree $\delta$,
$$\gamma_{oitR}(G)\leq 2\alpha(G)+\gamma(G)-\delta+1.$$
\end{theorem}

The bound above is tight for the case of star graphs, and in connection with this fact, we pose the following question.\\

\noindent
\textbf{Open question:} Is it the case that $\gamma_{oitR}(G)= 2\alpha(G)+\gamma(G)-\delta+1$ if and only if $G$ is a star graph?

\section{Exact formulas for some families of graphs}

This section is centered into giving the exact value of the outer-independent total Roman domination number of some significant families of graphs, that have been frequently studied in the literature in connection with several domination related invariants.

\subsection{Sierpi\'{n}ski graphs}

Sierpi\'nski graphs were introduced in \cite{Klavzar1997}. However, they were named in this way a little further in \cite{Klavzar2002}. For integers $n\geq 2$ and $p\geq 3$, the Sierpi\'nski graph $S_p^n$ is defined on the vertex set $\{0, 1,\ldots, p-1\}^n$, where two different vertices $(i_1,\ldots, i_n)$ and $(j_1, \ldots, j_n)$ are adjacent if and only if there exists an index $r$ in $\{1, \ldots, n\}$ such that the following follows.
\begin{itemize}
\item[(i)]   $i_t = j_t$, for $t=1,\ldots,r-1$;
\item[(ii)]  $i_r \neq j_r$; and
\item[(iii)] $i_t = j_r$ and $j_t = i_r$ for $t=r+1,\ldots,n$.
\end{itemize}
Three representative examples of Sierpi\'nski graphs are shown in Figure \ref{fig:Sierpinskis}. On the other hand, any Sierpi\'nski graph $S_p^n$ can be constructed by using an inductive manner. That is, we begin this construction with a clique of cardinality $p$ (a $p$-clique from now on), which can precisely be seen as $S_p^1$. Now, in order to construct $S_p^2$, we take $p$ copies of the $p$-clique and connect one with each other through an edge set in ``1-to-1'' correspondence with the edges of one $p$-clique. This procedure is repeated $n$ times to obtain our desired $S_p^n$. In general, we can recursively construct any $S_p^n$ by connecting $p$ copies of $S_p^{n-1}$ by using a set of $\frac{p(p-1)}{2}$ edges. The first two graphs of Figure \ref{fig:Sierpinskis} shows this process of construction, while using a clique of cardinality four. Although, we could include here several references to a large number of works which deals with Sierpi\'nski graph, it is not our goal to make so. Thus, for more information on such graph, we suggest the reader the nice survey \cite{Hinz2017}. We next compute the outer-independent total Roman domination number of Sierpi\'nski graphs.

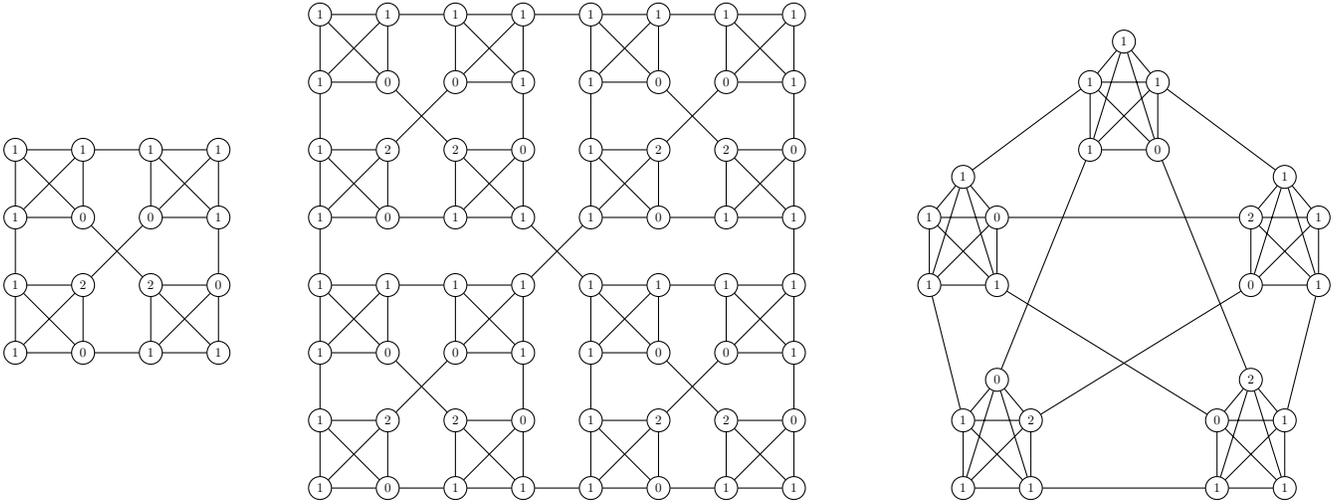
\begin{figure}[ht]
\centering
\begin{tikzpicture}[scale=.45, transform shape]

\node [draw, shape=circle] (b1) at  (3,-3) {1};
\node [draw, shape=circle] (b2) at  (1,-3) {1};
\node [draw, shape=circle] (b3) at  (1,-1) {2};
\node [draw, shape=circle] (b4) at  (3,-1) {0};

\node [draw, shape=circle] (c1) at  (-1,-3) {0};
\node [draw, shape=circle] (c2) at  (-3,-3) {1};
\node [draw, shape=circle] (c3) at  (-1,-1) {2};
\node [draw, shape=circle] (c4) at  (-3,-1) {1};

\node [draw, shape=circle] (d1) at  (-1,3) {1};
\node [draw, shape=circle] (d2) at  (-3,3) {1};
\node [draw, shape=circle] (d3) at  (-3,1) {1};
\node [draw, shape=circle] (d4) at  (-1,1) {0};

\node [draw, shape=circle] (e1) at  (3,3) {1};
\node [draw, shape=circle] (e2) at  (1,3) {1};
\node [draw, shape=circle] (e3) at  (1,1) {0};
\node [draw, shape=circle] (e4) at  (3,1) {1};

\node [draw, shape=circle] (b11) at  (20,-7) {1};
\node [draw, shape=circle] (b21) at  (18,-7) {1};
\node [draw, shape=circle] (b31) at  (18,-5) {2};
\node [draw, shape=circle] (b41) at  (20,-5) {0};
\node [draw, shape=circle] (c11) at  (16,-7) {0};
\node [draw, shape=circle] (c21) at  (14,-7) {1};
\node [draw, shape=circle] (c31) at  (14,-5) {1};
\node [draw, shape=circle] (c41) at  (16,-5) {2};
\node [draw, shape=circle] (d11) at  (16,-3) {0};
\node [draw, shape=circle] (d21) at  (14,-3) {1};
\node [draw, shape=circle] (d31) at  (14,-1) {1};
\node [draw, shape=circle] (d41) at  (16,-1) {1};
\node [draw, shape=circle] (e11) at  (20,-3) {1};
\node [draw, shape=circle] (e21) at  (18,-3) {0};
\node [draw, shape=circle] (e31) at  (18,-1) {1};
\node [draw, shape=circle] (e41) at  (20,-1) {1};

\node [draw, shape=circle] (b12) at  (12,-7) {1};
\node [draw, shape=circle] (b22) at  (10,-7) {1};
\node [draw, shape=circle] (b32) at  (10,-5) {2};
\node [draw, shape=circle] (b42) at  (12,-5) {0};
\node [draw, shape=circle] (c12) at  (8,-7) {0};
\node [draw, shape=circle] (c22) at  (6,-7) {1};
\node [draw, shape=circle] (c32) at  (6,-5) {1};
\node [draw, shape=circle] (c42) at  (8,-5) {2};
\node [draw, shape=circle] (d12) at  (8,-3) {0};
\node [draw, shape=circle] (d22) at  (6,-3) {1};
\node [draw, shape=circle] (d32) at  (6,-1) {1};
\node [draw, shape=circle] (d42) at  (8,-1) {1};
\node [draw, shape=circle] (e12) at  (12,-3) {1};
\node [draw, shape=circle] (e22) at  (10,-3) {0};
\node [draw, shape=circle] (e32) at  (10,-1) {1};
\node [draw, shape=circle] (e42) at  (12,-1) {1};

\node [draw, shape=circle] (b13) at  (12,7) {1};
\node [draw, shape=circle] (b23) at  (10,7) {1};
\node [draw, shape=circle] (b33) at  (10,5) {0};
\node [draw, shape=circle] (b43) at  (12,5) {1};
\node [draw, shape=circle] (c13) at  (8,7) {1};
\node [draw, shape=circle] (c23) at  (6,7) {1};
\node [draw, shape=circle] (c33) at  (6,5) {1};
\node [draw, shape=circle] (c43) at  (8,5) {0};
\node [draw, shape=circle] (d13) at  (8,3) {2};
\node [draw, shape=circle] (d23) at  (6,3) {1};
\node [draw, shape=circle] (d33) at  (6,1) {1};
\node [draw, shape=circle] (d43) at  (8,1) {0};
\node [draw, shape=circle] (e13) at  (12,3) {0};
\node [draw, shape=circle] (e23) at  (10,3) {2};
\node [draw, shape=circle] (e33) at  (10,1) {1};
\node [draw, shape=circle] (e43) at  (12,1) {1};

\node [draw, shape=circle] (b14) at  (20,7) {1};
\node [draw, shape=circle] (b24) at  (18,7) {1};
\node [draw, shape=circle] (b34) at  (18,5) {0};
\node [draw, shape=circle] (b44) at  (20,5) {1};
\node [draw, shape=circle] (c14) at  (16,7) {1};
\node [draw, shape=circle] (c24) at  (14,7) {1};
\node [draw, shape=circle] (c34) at  (14,5) {1};
\node [draw, shape=circle] (c44) at  (16,5) {0};
\node [draw, shape=circle] (d14) at  (16,3) {2};
\node [draw, shape=circle] (d24) at  (14,3) {1};
\node [draw, shape=circle] (d34) at  (14,1) {1};
\node [draw, shape=circle] (d44) at  (16,1) {0};
\node [draw, shape=circle] (e14) at  (20,3) {0};
\node [draw, shape=circle] (e24) at  (18,3) {2};
\node [draw, shape=circle] (e34) at  (18,1) {1};
\node [draw, shape=circle] (e44) at  (20,1) {1};

\node [draw, shape=circle] (f1) at  (25,-7) {1};
\node [draw, shape=circle] (f2) at  (27,-7) {1};
\node [draw, shape=circle] (f3) at  (25,-5) {1};
\node [draw, shape=circle] (f4) at  (27,-5) {2};
\node [draw, shape=circle] (f5) at  (26,-3.8) {0};

\node [draw, shape=circle] (g1) at  (32.5,-7) {1};
\node [draw, shape=circle] (g2) at  (34.5,-7) {1};
\node [draw, shape=circle] (g3) at  (32.5,-5) {0};
\node [draw, shape=circle] (g4) at  (34.5,-5) {1};
\node [draw, shape=circle] (g5) at  (33.5,-3.8) {2};

\node [draw, shape=circle] (h1) at  (24,-1) {1};
\node [draw, shape=circle] (h2) at  (26,-1) {1};
\node [draw, shape=circle] (h3) at  (24,1) {1};
\node [draw, shape=circle] (h4) at  (26,1) {0};
\node [draw, shape=circle] (h5) at  (25,2.2) {1};

\node [draw, shape=circle] (i1) at  (33.5,-1) {0};
\node [draw, shape=circle] (i2) at  (35.5,-1) {1};
\node [draw, shape=circle] (i3) at  (33.5,1) {2};
\node [draw, shape=circle] (i4) at  (35.5,1) {1};
\node [draw, shape=circle] (i5) at  (34.5,2.2) {1};

\node [draw, shape=circle] (j1) at  (28.75,3) {1};
\node [draw, shape=circle] (j2) at  (30.75,3) {0};
\node [draw, shape=circle] (j3) at  (28.75,5) {1};
\node [draw, shape=circle] (j4) at  (30.75,5) {1};
\node [draw, shape=circle] (j5) at  (29.75,6.2) {1};

\draw(f1)--(f2)--(f3)--(f4)--(f5)--(f1)--(f3)--(f5)--(f2)--(f4)--(f1);
\draw(g1)--(g2)--(g3)--(g4)--(g5)--(g1)--(g3)--(g5)--(g2)--(g4)--(g1);
\draw(h1)--(h2)--(h3)--(h4)--(h5)--(h1)--(h3)--(h5)--(h2)--(h4)--(h1);
\draw(i1)--(i2)--(i3)--(i4)--(i5)--(i1)--(i3)--(i5)--(i2)--(i4)--(i1);
\draw(j1)--(j2)--(j3)--(j4)--(j5)--(j1)--(j3)--(j5)--(j2)--(j4)--(j1);

\draw(f2)--(g1);
\draw(f4)--(i1);
\draw(f3)--(h1);
\draw(f5)--(j1);
\draw(g3)--(h2);
\draw(g5)--(j2);
\draw(g4)--(i2);
\draw(h5)--(j3);
\draw(h4)--(i3);
\draw(i5)--(j4);

\draw(b1)--(b2)--(b3)--(b4)--(b1)--(b3);
\draw(b2)--(b4);
\draw(c1)--(c2)--(c3)--(c4)--(c1)--(c3);
\draw(c2)--(c4);
\draw(d1)--(d2)--(d3)--(d4)--(d1)--(d3);
\draw(d2)--(d4);
\draw(e1)--(e2)--(e3)--(e4)--(e1)--(e3);
\draw(e2)--(e4);
\draw(b2)--(c1);
\draw(c4)--(d3);
\draw(d1)--(e2);
\draw(e4)--(b4);
\draw(b3)--(d4);
\draw(c3)--(e3);

\draw(b11)--(b21)--(b31)--(b41)--(b11)--(b31);
\draw(b21)--(b41);
\draw(c11)--(c21)--(c31)--(c41)--(c11)--(c31);
\draw(c21)--(c41);
\draw(d11)--(d21)--(d31)--(d41)--(d11)--(d31);
\draw(d21)--(d41);
\draw(e11)--(e21)--(e31)--(e41)--(e11)--(e31);
\draw(e21)--(e41);
\draw(b21)--(c11);
\draw(c31)--(d21);
\draw(d41)--(e31);
\draw(e11)--(b41);
\draw(b31)--(d11);
\draw(c41)--(e21);

\draw(b12)--(b22)--(b32)--(b42)--(b12)--(b32);
\draw(b22)--(b42);
\draw(c12)--(c22)--(c32)--(c42)--(c12)--(c32);
\draw(c22)--(c42);
\draw(d12)--(d22)--(d32)--(d42)--(d12)--(d32);
\draw(d22)--(d42);
\draw(e12)--(e22)--(e32)--(e42)--(e12)--(e32);
\draw(e22)--(e42);
\draw(b22)--(c12);
\draw(c32)--(d22);
\draw(d42)--(e32);
\draw(e12)--(b42);
\draw(b32)--(d12);
\draw(c42)--(e22);

\draw(b13)--(b23)--(b33)--(b43)--(b13)--(b33);
\draw(b23)--(b43);
\draw(c13)--(c23)--(c33)--(c43)--(c13)--(c33);
\draw(c23)--(c43);
\draw(d13)--(d23)--(d33)--(d43)--(d13)--(d33);
\draw(d23)--(d43);
\draw(e13)--(e23)--(e33)--(e43)--(e13)--(e33);
\draw(e23)--(e43);
\draw(b23)--(c13);
\draw(c33)--(d23);
\draw(d43)--(e33);
\draw(e13)--(b43);
\draw(b33)--(d13);
\draw(c43)--(e23);

\draw(b14)--(b24)--(b34)--(b44)--(b14)--(b34);
\draw(b24)--(b44);
\draw(c14)--(c24)--(c34)--(c44)--(c14)--(c34);
\draw(c24)--(c44);
\draw(d14)--(d24)--(d34)--(d44)--(d14)--(d34);
\draw(d24)--(d44);
\draw(e14)--(e24)--(e34)--(e44)--(e14)--(e34);
\draw(e24)--(e44);
\draw(b24)--(c14);
\draw(c34)--(d24);
\draw(d44)--(e34);
\draw(e14)--(b44);
\draw(b34)--(d14);
\draw(c44)--(e24);

\draw(c21)--(b12);
\draw(d32)--(d33);
\draw(b13)--(c24);
\draw(e44)--(e41);
\draw(d31)--(e43);
\draw(e42)--(d34);


\end{tikzpicture}
\caption{The graphs $S_4^2$, $S_4^3$ and $S_5^2$ (from left to right), and a labeling for an OITRDF of minimum weight in each graph.}\label{fig:Sierpinskis}
\end{figure}

\begin{theorem}
For any Sierpi\'{n}ski graph $S_p^n$, $$\gamma_{oitR}(S_p^n)=p\left\lceil\frac{p^{n-1}}{2}\right\rceil+(p-1)\left\lfloor\frac{p^{n-1}}{2}\right\rfloor.$$
\end{theorem}

\begin{proof}
We first note that in each subgraph of $S_p^n$ isomorphic to $K_p$, only one vertex can belong to $V_0$ for any $\gamma_{oitR}(S_p^n)$-function $f(V_0,V_1,V_2)$. Moreover, each of these vertices in $V_0$ needs to have a neighbor in $V_2$. Since every vertex of each copy of $K_p$ in $S_p^n$ has at most one neighbor not in the same copy of $K_p$, this means that a vertex of $V_2$ can have at most two neighbors in $V_0$, one from the same copy it belongs, one from another copy. If there is a vertex labeled 2 under $f$ in a copy $Q$ of $K_p$ in $S_p^n$, then $f(V(Q))\ge p$. On the contrary, if there is no vertex labeled 2 under $f$ in a copy $Q'$ of $K_p$ in $S_p^n$, then $f(V(Q'))\ge p-1$. Note that $S_p^n$ contains $p^{n-1}$ disjoint subgraphs isomorphic to $K_p$ (the copies of $K_p$). Hence, since there could be at most one vertex in each copy of $K_p$ labeled 0 under $f$ and each vertex in $V_2$ can have at most two neighbors in $V_0$, we deduce that there at least $\left\lceil\frac{p^{n-1}}{2}\right\rceil$ copies of $K_p$ containing a vertex labeled 2 under $f$. Assume $Q_1,\dots,Q_{r}$, with $r\ge \left\lceil\frac{p^{n-1}}{2}\right\rceil$, are the copies of $K_p$ containing a vertex with label 2. This leaves $p^{n-1}-r$ copies of $K_p$ to have no vertex with label 2, denoted by $Q'_1,\dots,Q'_{p^{n-1}-r}$. This leads to the following.
$$\gamma_{oitR}(S_p^n)=\omega(f)=\sum_{i=1}^{r}f(V(Q_i))+\sum_{i=1}^{p^{n-1}-r}f(V(Q'_i))\ge pr+(p-1)(p^{n-1}-r)=r+(p-1)p^{n-1}.$$
Since $r\ge \left\lceil\frac{p^{n-1}}{2}\right\rceil$, we obtain that
$$\gamma_{oitR}(S_p^n)\ge (p-1)p^{n-1}+\left\lceil\frac{p^{n-1}}{2}\right\rceil=p\left\lceil\frac{p^{n-1}}{2}\right\rceil+(p-1)\left\lfloor\frac{p^{n-1}}{2}\right\rfloor.$$
On the other hand, we shall construct a $\gamma_{oitR}(S_p^n)$-function of weight $p\left\lceil\frac{p^{n-1}}{2}\right\rceil+(p-1)\left\lfloor\frac{p^{n-1}}{2}\right\rfloor$ in the following way. We will proceed iteratively. That is, we begin with a function $f_2$ for the graph $S_p^2$ with a specific structure which we require for our purposes. One possible labeling, for a function $f_2$ for the cases $p=4$ and $p=5$, is shown in Figure \ref{fig:Sierpinskis}. There are two requirements that we need for our function:
\begin{itemize}
  \item[{\rm (i)}] There are exactly $p$ vertices labeled with 0, and they are lying in ``interior'' positions of the drawing of $S_p^2$. Notice that this latter requirement can be always done unless $p=4$.
  \item[{\rm (ii)}] There are $\left\lceil\frac{p^{n-1}}{2}\right\rceil$ vertices labeled with 2, and they have at least one and at most two neighbors labeled with 0 that belong to two different copies of $K_p$. If $p$ is even, each vertex labeled 2 has exactly two neighbors labeled 0. This could happen also is $p$ is odd, but not necessarily.
\end{itemize}
Notice that $f_2$ has weight $p\left\lceil\frac{p^{n-1}}{2}\right\rceil+(p-1)\left\lfloor\frac{p^{n-1}}{2}\right\rfloor$, and so, it is a $\gamma_{oitR}(S_p^n)$-function.

Now, for the graph $S_p^3$, we need $p$ copies of the graph $S_p^2$. Then, we use in each copy, the labeling used in the graph $S_p^2$. Since the vertices labeled with 0 are ``interiors'' in their corresponding copies, they will not have a neighbor in another copy, and clearly they remain having a neighbor labeled 2. Thus, this labeling is taken as a function $f_3$ which is indeed a $\gamma_{oitR}(S_p^n)$-function of weight $p\left\lceil\frac{p^{n-1}}{2}\right\rceil+(p-1)\left\lfloor\frac{p^{n-1}}{2}\right\rfloor$. We note that, if $p=4$, then, although the vertices labeled 0 are not ``interior'', the property for a vertex in a copy of $S_4^2$ of not having a neighbor in another copy remains. Thus, the same conclusion can be deduced for $S_4^3$. A repetition of this process will always produce a $\gamma_{oitR}(S_p^n)$-function of weight $p\left\lceil\frac{p^{n-1}}{2}\right\rceil+(p-1)\left\lfloor\frac{p^{n-1}}{2}\right\rfloor$ for any values of $p$ and $n$. This completes the proof.
\end{proof}

\subsection{Circulant graphs}

Let $\mathbb{Z}_n$ be the additive group of integers modulo $n$ and let $M\subset \mathbb{Z}_n$, such that, $i\in M$ if and only if $-i\in M$. A circulant graph $G$ of order $n$ with respect to $M$ is a graph $G$ constructed as follows. \textcolor[rgb]{0.00,0.00,0.00}{The vertices of $G$} are the elements of $\mathbb{Z}_n$ and $(i,j)$ is an edge of $G$ if and only if $j-i\in M$. With such notation above, we see that a cycle is precisely a circulant graph when $M=\{-1,1\}$, while a complete graph is also a circulant graph with $M=\mathbb{Z}_n$. Now, in order to simplify the notation, we shall use $C(n,k)$, $0<k\leq\left\lfloor\frac{n}{2}\right\rfloor$, instead of $CR(n,\{-k,-k+1,...,-1,1,2,...,k\})$. Moreover, we shall next assume that $V(C(n,k))=\textcolor[rgb]{0.00,0.00,0.00}{\{v_0,v_1, \ldots, v_{n-1}\}}$, such that $v_i$ is adjacent to $v_{i+j}$ and $v_{i-j}$ with \textcolor[rgb]{0.00,0.00,0.00}{$j=1, \ldots, k$}, where the subscripts are taken modulo $n$. In order to compute $\gamma_{oitR}(C(n,k))$, we shall need the following result (which could be a known one, but we know no reference in which this appears).

\begin{lemma}\label{lem:indep-circ}
For any integers $n$ and $2\leq k\leq \lfloor n/2 \rfloor$, $\beta(C(n,k))=\left\lfloor\frac{n}{k+1}\right\rfloor$.
\end{lemma}

\begin{proof}
Let $S=\{v_i\,:\,i\equiv 0\,\mbox{(\textcolor[rgb]{0.00,0.00,0.00}{k+1})}\}\subset V(C(n,k))$. Since $v_i$ is not adjacent to every vertex $v_j$ such that $i\equiv j (\textcolor[rgb]{0.00,0.00,0.00}{k+1})$, it is clear that $S$ is an independent set of $C(n,k)$, and so, $\beta(C(n,k))\ge |S|=\left\lfloor\frac{n}{k+1}\right\rfloor$.

On the other hand, let $S'$ be a $\beta(C(n,k))$-set. For every $i\in \{0,\dots,n-1\}$, we shall consider the set $A_i=\{v_i,v_{i+1},\dots,v_{i+k}\}$. Notice that $|S'\cap A_i|\le 1$ for every $i\in \{0,\dots,n-1\}$. Thus, we have the following.
$$\beta(C(n,k))=\frac{1}{k+1}\sum_{i=0}^{n-1}|S'\cap A_i|\le \frac{n}{k+1}$$
Since $\beta(C(n,k))$ is an integer, it must happen $\beta(C(n,k))\le \left\lfloor\frac{n}{k+1}\right\rfloor$, which completes the proof.
\end{proof}

\begin{theorem}\label{th:circulant}
For any integers $n$ and $2\leq k\leq \lfloor n/2 \rfloor$,
$$\gamma_{oitR}(C(n,k))=n-\left\lfloor\frac{\left\lfloor \frac{n}{k+1} \right\rfloor}{2}\right\rfloor.$$
\end{theorem}

\begin{proof}
Let $f(V_0,V_1,V_2)$ be a function on $C(n,k)$ defined as follows

 $$V_0=\bigcup_{i=0}^{\lfloor n/(k+1)\rfloor}\{v_{i(k+1)}\}, \hspace{.4cm} V_2=\bigcup_{i=0}^{\lfloor n/(k+1)\rfloor}\{v_{i(k+1)+1}\} \hspace{.3cm} \text{ and } \hspace{.3cm} V_1=V(C(n,k))\setminus (V_1\cup V_2).$$ Notice that $f$ is an OITRDF on $C(n,k)$. Hence, $\gamma_{oitR}(C(n,k))\leq \omega(f)=|V_1|+2|V_2|=n-\left\lfloor\frac{\left\lfloor \frac{n}{k+1} \right\rfloor}{2}\right\rfloor.$

In consequence, we only need to prove that $\gamma_{oitR}(C(n,k))\geq n-\left\lfloor\frac{\left\lfloor \frac{n}{k+1} \right\rfloor}{2}\right\rfloor.$ Let $f(V_0,V_1,V_2)$ be a $\gamma_{oitR}(C(n,k))$-function. Notice that $|V_0|\leq 2|V_2|$ which leads to $\left\lceil\frac{|V_0|}{2}\right\rceil\le |V_2|$. We consider the following.
\begin{align*}
\label{circ-chain}
  \gamma_{oitR}(C(n,k))=\omega(f) & =2|V_2|+|V_1| \\
   & = n -|V_0|+|V_2|  \\
   & \ge n-|V_0|+\left\lceil\frac{|V_0|}{2}\right\rceil \\
   & =n-\left\lfloor\frac{|V_0|}{2}\right\rfloor.
\end{align*}
Since $|V_0|\le \beta(C(n,k))=\left\lfloor \frac{n}{k+1} \right\rfloor$ (the last part by Lemma \ref{lem:indep-circ}), we deduce that $\gamma_{oitR}(C(n,k))\ge n-\left\lfloor\frac{\left\lfloor \frac{n}{k+1} \right\rfloor}{2}\right\rfloor$, which completes the proof.
\end{proof}

\subsection{Products of complete graphs}

Interest in studies on products of complete graphs might has come from the well known Hamming graphs $H(d,q)$, which are indeed the Cartesian product of $d$ complete graphs $K_q$. For our exposition, we simply begin with studying the case in which $d=2$, but the research could be continued to larger values of $d$. In addition, we also study other products of complete graphs.

Let $V(K_r)=\{u_1,\dots,u_r\}$ and $V(K_s)=\{v_1,\dots,v_s\}$. We now consider the outer-independent total Roman domination number of the four standard products (Cartesian-$\square$, direct-$\times$, strong-$\boxtimes$ and lexicographic-$\circ$, according to \cite{Hammack2011}) of complete graphs. First it is clear that the strong and lexicographic product of complete graphs result in a complete graph as well. Thus, $\gamma_{oitR}(K_r\boxtimes K_s)=\gamma_{oitR}(K_r\circ K_s)=rs$. We next give also exact formulas for the remaining two standard products. For information, definitions and usual terminology on product graphs we suggest the very complete book \cite{Hammack2011}.

\begin{theorem}
For any integers $r,s$ with $2\le r\le s$, $\gamma_{oitR}(K_r\square K_s)=rs-\left\lfloor\frac{r}{2}\right\rfloor$.
\end{theorem}

\begin{proof}
Let $f(V_0,V_1,V_2)$ be a function on $K_r\square K_s$ such that $V_0=\{(u_i,v_i)\,:\,1\le i\le r\}$, $V_2=\{(u_{2i},v_{2i-i})\,:\,1\le i\le \left\lceil r/2\right\rceil\}$ and $V_1=V(K_r\square K_s)\setminus (V_0\cup V_2)$. We readily see that $V_0$ is independent, that $V_1\cup V_2$ is a total dominating set, and that every vertex in $V_0$ is adjacent to a vertex in $V_2$. Thus, $f$ is an OITRDF on $K_r\square K_s$, and so, $\gamma_{oitR}(K_r\square K_s)\le \omega(f)=rs-r+\left\lceil\frac{r}{2}\right\rceil=rs-\left\lfloor\frac{r}{2}\right\rfloor$.

On the other hand, let $f'(V'_0,V'_1,V'_2)$ be a $\gamma_{oitR}(K_r\square K_s)$-function. From now, the proof follows along the lines of the second part of the proof of Theorem \ref{th:circulant}. That is, since similarly $\left\lceil\frac{|V'_0|}{2}\right\rceil\le |V'_2|$ and $|V'_0|\le \beta(K_r\square K_s)=r$, we deduce that $\gamma_{oitR}(K_r\square K_s)\ge rs-\left\lfloor\frac{r}{2}\right\rfloor$.
\end{proof}

\begin{theorem}
For any integers $r,s$ with $2\le r\le s$, $\gamma_{oitR}(K_r\times K_s)=s(r-1)+2$.
\end{theorem}

\begin{proof}
Let $f(V_0,V_1,V_2)$ be a function on $K_r\times K_s$ such that $V_0=\{\textcolor[rgb]{0.00,0.00,0.00}{(u_1,v_i)}\,:\,1\le i\le s\}$, $V_2=\{\textcolor[rgb]{0.00,0.00,0.00}{(u_2,v_1)},(u_2,v_2)\}$ and $V_1=V(K_r\times K_s)\setminus (V_0\cup V_2)$. It can be easily noted that $f$ is an OITRDF on $K_r\times K_s$. Thus $\gamma_{oitR}(K_r\times K_s)\le s(r-1)+2$.

Now, let $f'(V'_0,V'_1,V'_2)$ be a $\gamma_{oitR}(K_r\times K_s)$-function. Since $2\le |V'_0|\le \beta(K_r\times K_s)=s$, it follows $|V'_2|\ge 2$, and we have that
$\gamma_{oitR}(K_r\times K_s)=\omega(f)=2|V'_2|+|V'_1| = rs -|V'_0|+|V'_2|\ge rs-s+2=s(r-1)+2$, which completes the proof.
\end{proof}

\section*{Acknowledgments}

The last two authors (Dorota Kuziak and Ismael G. Yero) have been partially supported by ``Junta de Andaluc\'ia'', FEDER-UPO Research and Development Call, reference number UPO-1263769.

\end{document}